\documentclass {amsart}

\usepackage{mathrsfs}
\usepackage{graphicx,verbatim}
\usepackage{amscd}
\usepackage{float}
\usepackage{amsmath}
\usepackage{amssymb}

\newcommand{\bbP}{\mathbb{P}}

\newcommand{\bbF}{\mathbb{F}}
\newcommand{\bbZ}{\mathbb{Z}}

\newcommand{\calM}{\mathcal{M}}
\newcommand{\calS}{\mathcal{S}}

\newcommand{\isomap}{\stackrel{\thicksim}{\longrightarrow}}
\newcommand{\calO}{\mathcal{O}}

\newtheorem{theorem}{Theorem}
\newtheorem{lemma}[theorem]{Lemma}
\newtheorem{proposition}[theorem]{Proposition}
\newtheorem{definition}[theorem]{Definition}
\newtheorem{corollary}[theorem]{Corollary}

\begin{document}


\title[Moduli of bundles II]{Moduli of bundles over rational surfaces and elliptic curves
II: non-simply laced cases }


\author{Naichung Conan Leung}


\address{the Institute of Mathematical Sciences and Department of Mathematics,
         The Chinese University of Hong Kong,
         Shatin, N.T., Hong Kong}

\email{leung@ims.cuhk.edu.hk}


\author{Jiajin Zhang}
\address{Department of Mathematics, Sichuan University, Chengdu, 610000, P.R. China}
\email{jjzhang@scu.edu.cn}
\curraddr{Institute of Mathematics, Johannes Gutenberg-University
Mainz, Mainz, 55099, Germany} \email{zhangji@uni-mainz.de}





\begin{abstract}
   For any non-simply laced Lie group $G$ and elliptic curve $\Sigma$, we show that the moduli space of
flat $G$ bundles over $\Sigma$ can be identified with the moduli
space of rational surfaces with $G$-configurations which contain
$\Sigma$ as an anti-canonical curve. We also construct
$Lie(G)$-bundles over these surfaces. The corresponding results for
simply laced groups were obtained by the authors in another paper.
Thus we have established a natural identification for these two
kinds of moduli spaces for any Lie group $G$.
\end{abstract}

\maketitle

\section{Introduction}\label{introduction}

     There is a classical identification for the moduli space of flat
$E_n$ bundles over a fixed elliptic curve $\Sigma$ and the moduli
space of del Pezzo surfaces of degree $9-n$ containing $\Sigma$
as an anti-canonical curve, see \cite{Donagi}\cite{Donagi2}\cite{FM1}\cite{FMW1}.
For other simply laced Lie group $G$, we also obtained in \cite{LZ} an identification
for the moduli space of flat $G$ bundles over a fixed
elliptic curve $\Sigma$ and the moduli space of the pairs
$(S,\Sigma)$ where $S$ is a rational surface (called $ADE$-surface in \cite{LZ})
containing $\Sigma$ as an anti-canonical curve. In this paper, we extend the above
identification to non-simply laced cases. Therefore we establish a one-to-one correspondence
between flat $G$ bundles over a fixed elliptic curve $\Sigma$ and rational surfaces with
$\Sigma$ as an anti-canonical curve for Lie groups of {\it all} types.

      A non-simply laced Lie group $G$ is uniquely determined by a
simply laced Lie group $G'$ and its outer automorphism group. Hence
it is natural to apply the previous results for the simply laced cases in \cite{LZ} to the current situation.
Similar to simply-laced cases, we can define $G$-{\it
surfaces} and {\it rational surfaces with $G$-configurations} (see
Definition~\ref{Bn-config}, ~\ref{Cn-config}, ~\ref{G2-config},
~\ref{F4-config}). Our main result is the following:

   \begin{theorem} Let $\Sigma$ be a fixed elliptic curve with
   identity $0\in \Sigma$,  $G$ be any simple, compact, simply
   connected and non-simply laced  Lie group. Denote $\calS(\Sigma,G)$ the  moduli space of the pairs
   $(S,\Sigma)$, where $S$ is a $G$-surface such that $\Sigma\in|-K_S|$. Denote $\calM^G_{\Sigma}$ the moduli space
   of flat $G$-bundles over $\Sigma$. Then we have\par
     (i) $\calS(\Sigma,G)$ can be
   embedded into $\calM^G_{\Sigma}$  as an open dense subset.\par
     (ii) This embedding can be extended to an isomorphism from $\overline{\calS(\Sigma,G)}$ onto $\calM^G_{\Sigma}$
     by including all rational surfaces with $G$-configurations, and this
   gives us a natural and explicit compactification $\overline{\calS(\Sigma,G)}$ of $\calS(\Sigma,G)$.
   \end{theorem}

      This paper is motivated by some duality in physics.
   When $G=E_n$ is a simple subgroup of $E_8\times E_8$, these $G$ bundles are related
   to the duality between  $F$-theory and string theory. Among other things, this duality predicts the moduli of
   flat $E_n$ bundles over a fixed elliptic curve $\Sigma$ can be identified with the
   moduli of del Pezzo surfaces with fixed anti-canonical curve $\Sigma$. For details, one can consult  \cite{Donagi}
   \cite{Donagi2}  \cite{FMW1}. Our result can be considered as a test of above duality for other Lie groups.\par

      As an application, this identification provides us with an intuitive explanation
  for $\calM^G_{\Sigma}$. We also provide an interesting geometric realization of
  root system theory, and we can see very clearly how the Weyl group acts on the (marked)
  moduli space of flat $G$-bundles over $\Sigma$.

      In the following, we illustrate briefly via pictures what $G$-configurations and $G$-surfaces are
   in each case and compare it  with the corresponding case where $G'$ is simply-laced.

\subsection{$B_n$-configurations as special
$D_{n+1}$-configurations}\label{Bn-intro}
      In these cases we consider rational surfaces with fibration structure and
   a fixed smooth anti-canonical curve $\Sigma$.
      A $B_n$-configuration comes from a $D_{n+1}$-configuration. Roughly speaking,
      by saying a rational surface $S$ has a $D_{n+1}$-configuration
      $(l_1,\cdots,l_{n+1})$, we mean that $S$ can be
      considered as a blow-up of $\bbF_1$ (a {\it Hirzebruch
      surface}) at $n+1$ points on $\Sigma\in |-K_{\bbF_1}|$, such
      that $l_1,\cdots,l_{n+1}$ are the corresponding exceptional
      classes \cite{LZ}. When these blown up points are {\it in general position}, $S$ is called a
      $G=D_{n+1}$-surface.
      See Figure 1 for a surface with a $D_{n+1}$-configuration.

         Given a surface $S$ with a $D_{n+1}$-configuration $\zeta=(l_1,\cdots,l_{n+1})$,
      if  it satisfies the condition $x_1=l_1\cap\Sigma$ is the identity element $0$ of the elliptic curve $\Sigma$,
      then $\zeta$ is a $B_n$-configuration on
      $S$ (Definition~\ref{Bn-config}). If all blown up points but $x_1$ are {\it in general position}, $S$ is called a
      $B_n$-surface. See Figure 2 for a surface with a $B_n$-configuration.

\begin{figure}[H]\label{figure-D}
\centering
\includegraphics{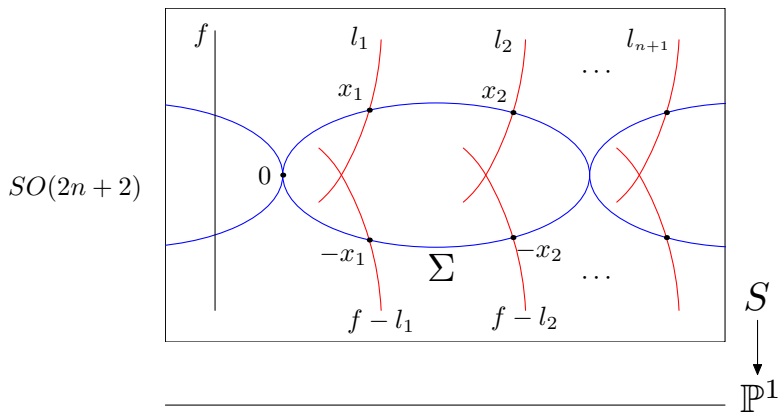}
\caption{A surface with a $D_{n+1}$-configuration $(l_1,\cdots,l_{n+1})$}
\end{figure}

\begin{figure}[H]
\centering
\includegraphics{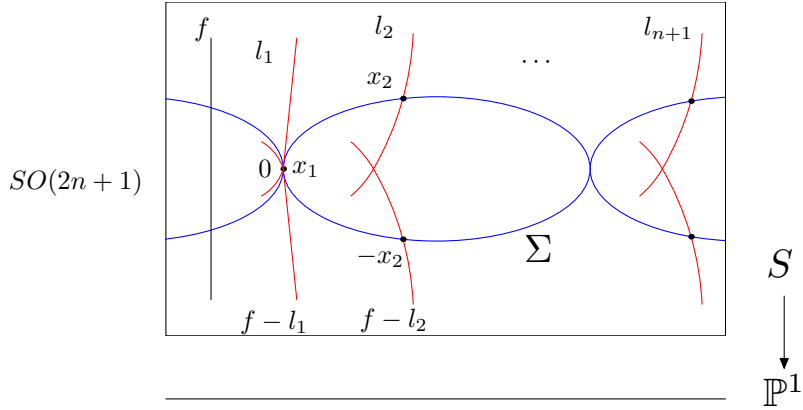}
\caption{A surface with a $B_n$-configuration $(l_1,l_2,\cdots,l_{n+1})$, where $x_1=l_1\cap\Sigma=0$}
\end{figure}

\subsection{$C_n$-configurations as special
$A_{2n-1}$-configurations}\label{Cn-intro}

         In these cases, we consider rational surfaces with fibration and section structure
      and a fixed smooth anti-canonical curve $\Sigma$.\par
      A $C_n$-configuration comes from an $A_{2n-1}$-configuration.
      By saying a rational surface $S$ has an $A_{2n-1}$-configuration
      $(l_1,\cdots,l_{2n})$, we mean that $S$ can be
      considered as a blow-up of $\bbF_1$  at $2n$ points on $\Sigma\in |-K_{\bbF_1}|$  which sum to zero, such
      that $l_1,\cdots,l_{2n}$ are the corresponding exceptional
      classes \cite{LZ}. When these blown up points are {\it in general position}, $S$ is called an
      $A_{2n-1}$-surface.
      See Figure 3 for a surface with an $A_{2n-1}$-configuration.

\begin{figure}[H]
\centering
\includegraphics{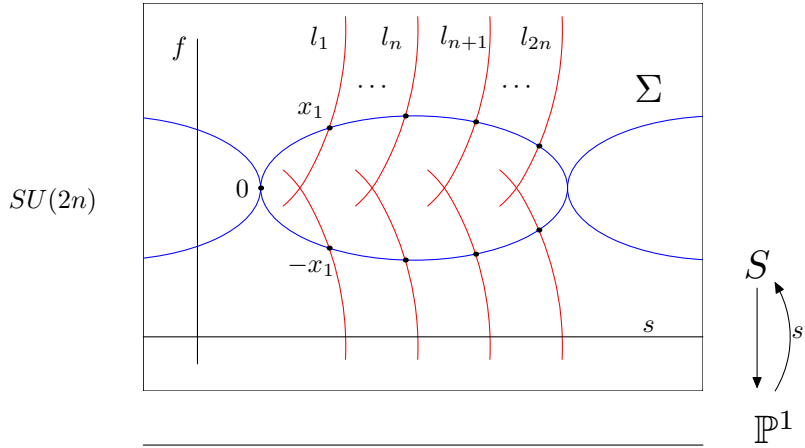}
\caption{A surface with an $A_{2n-1}$-configuration $(l_1,\cdots,l_{2n})$}
\end{figure}

         Given a surface $S$ with an $A_{2n-1}$-configuration $\zeta=(l_1,\cdots,l_{2n})$,
      if  it satisfies the condition $x_i=-x_{2n+1-i}$ with $x_i=l_i\cap\Sigma$, for $i=1,\cdots,n$,
      then $\zeta$ is called a $C_n$-configuration on
      $S$ (Definition~\ref{Cn-config}). If all blown up points are {\it in general position}, $S$ is called a
      $C_n$-surface. See Figure 4 for a surface with a $C_n$-configuration.

\begin{figure}[H]
\centering
\includegraphics{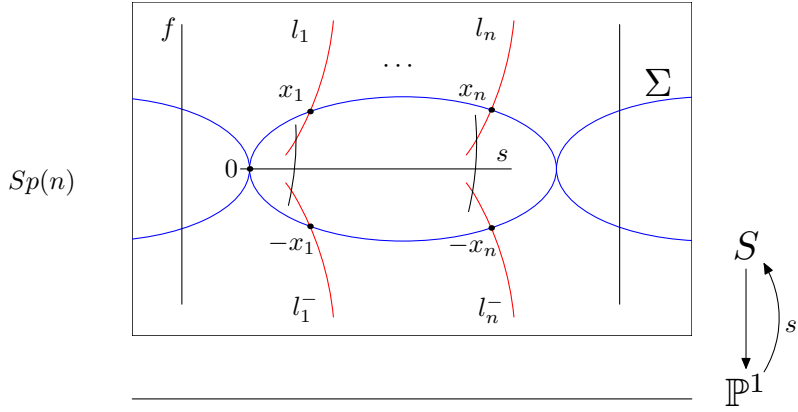}
\caption{A surface with a $C_n$-configuration $(l_1,\cdots,l_n,l_n^-,\cdots,l_1^-)$}
\end{figure}

\subsection{$G_2$-configurations as special
$D_{4}$-configurations}\label{G2-intro}

     In these cases we still consider rational surfaces with fibration
structure and a fixed smooth anti-canonical curve
$\Sigma$.\par
      A $G_2$-configuration comes from a $D_{4}$-configuration. We have seen what a $D_4$-configuration is
      from Subsection~\ref{Bn-intro}. Roughly speaking,
      by saying a rational surface $S$ has a $D_{4}$-configuration
      $(l_1,\cdots,l_{4})$, we mean that $S$ can be
      considered as a blow-up of $\bbF_1$  at $4$ points on $\Sigma\in |-K_{\bbF_1}|$, such
      that $l_1,\cdots,l_{4}$ are the corresponding exceptional
      classes \cite{LZ}. When these blown up points are {\it in general position}, $S$ is called a
      $G=D_{4}$-surface.
      See Figure 5 for a surface with a $D_{4}$-configuration.

\begin{figure}[H]
\centering
\includegraphics{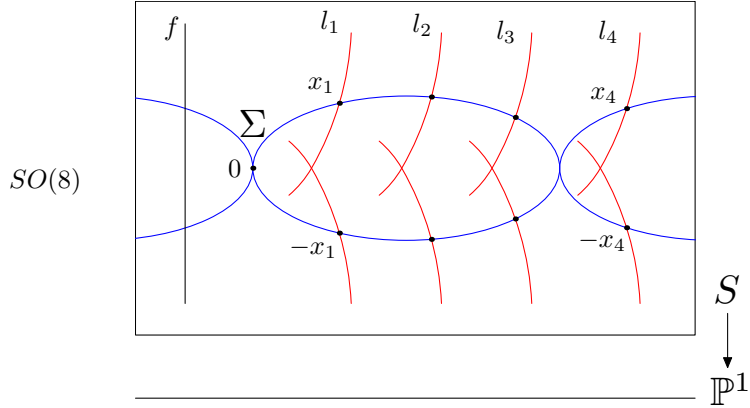}
\caption{A surface with a $D_4$-configuration $(l_1,\cdots,l_4)$}
\end{figure}

         Given a surface $S$ with a $D_{4}$-configuration $\zeta=(l_1,\cdots,l_{4})$,
      if  it satisfies these two conditions $x_1=0$ and $x_4=x_2+x_3$, where $x_i=l_i\cap\Sigma$,
      then $\zeta$ is called a $G_2$-configuration on
      $S$ (Definition~\ref{G2-config}). If all blown up points but $x_1$ are {\it in general position}, $S$ is called a
      $G_2$-surface. See Figure 6 for a surface with a $G_2$-configuration.

\begin{figure}[H]
\centering
\includegraphics{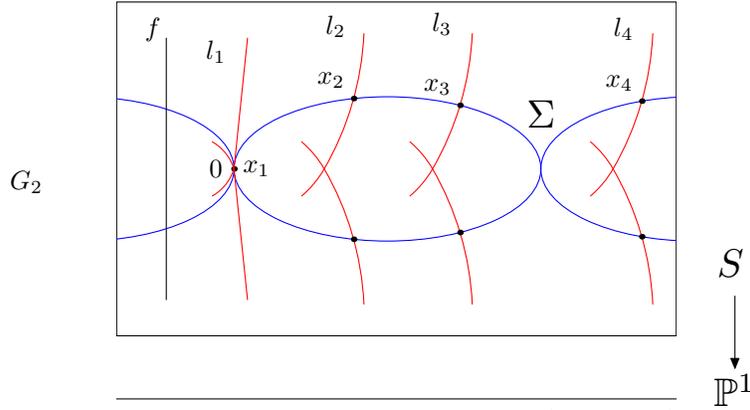}
\caption{A surface with a $G_2$-configuration $(l_1,l_2,l_3,l_4)$,
where $x_1=0$ and $x_4=x_2+x_3$ with $x_i=l_i\cap \Sigma$}
\end{figure}

\subsection{$F_4$-configurations as special
$E_{6}$-configurations}\label{F4-intro}

     In these cases we consider rational surfaces which are blow-ups of  the projective plane $\bbP^2$
     at $6$ points in almost general position, and which contain a fixed smooth anti-canonical curve
$\Sigma$ \cite{LZ}.\par
      An $F_4$-configuration comes from an $E_6$-configuration.
      Recall that by saying a rational surface $S$ has an $E_6$-configuration
      $(l_1,\cdots,l_{6})$, we mean that $S$ can be
      considered as a blow-up of $\bbP^2$  at $6$ points on $\Sigma\in |-K_{\bbP^2}|$, such
      that $l_1,\cdots,l_{6}$ are the corresponding exceptional
      classes. When these blown up points are {\it in general position}, $S$ is called an
      $E_6$-surface, which is in fact a cubic surface.
      See Figure 7 for a surface with an $E_6$-configuration.

\begin{figure}[H]
\centering
\includegraphics{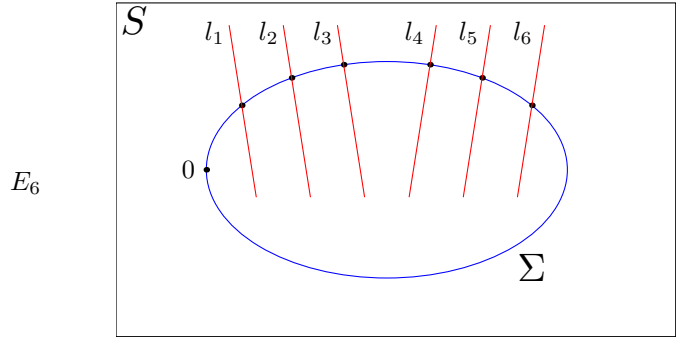}
\caption{A surface with an $E_6$-configuration $(l_1,\cdots,l_6)$}
\end{figure}

         Given a surface $S$ with an $E_6$-configuration $\zeta=(l_1,\cdots,l_{6})$,
      if  it satisfies the condition $x_1+x_6=x_2+x_5=x_3+x_4$, where $x_i=l_i\cap\Sigma$,
      then $\zeta$ is called an $F_4$-configuration on
      $S$ (Definition~\ref{F4-config}). If all blown up points are {\it in general position}, $S$ is called an
      $F_4$-surface. See Figure 8 for a surface with an $F_4$-configuration.

\begin{figure}[H]
\centering
\includegraphics{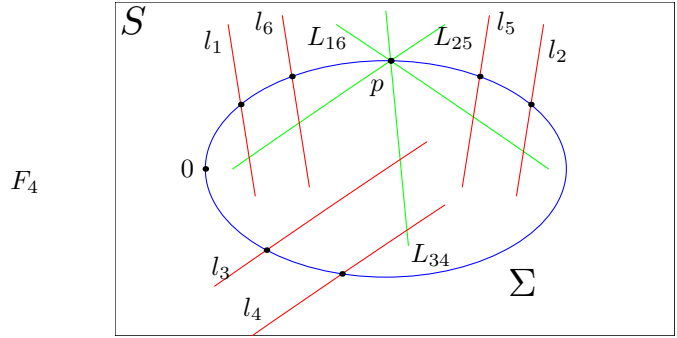}
\caption{A surface with an $F_4$-configuration $(l_1,\cdots,l_6)$
where three lines $L_{16}, L_{25}, L_{34}$ meet at $p\in\Sigma$,
or equivalently, $x_1+x_6=x_2+x_5=x_3+x_4$ with $x_i=l_i\cap\Sigma$}
\end{figure}

      Moreover, we can construct  $\mathcal {G}$$=\mbox{Lie}(G)$
   bundles over $S$ with a $G$-configuration. By restriction, we obtain $\mbox{Lie}(G)$ bundles over $\Sigma$.
   And we can also constructed some natural fundamental representation
   bundles over $\Sigma$ which have interesting geometric meanings, such that the Lie algebra
   bundles are the automorphism bundles of these
   representation bundles preserving certain algebraic structures.

In this paper, the notations are the same as those in \cite{LZ}. Let
$G$ be a compact, simple and simply-connected Lie group. We
denote\par
   $r(G)$: the rank of $G$;\par
   $R(G)$: the root system;\par
   $R_c(G)$: the coroot system;\par
   $W(G)$: the Weyl group;\par
   $\Lambda(G)$: the root lattice;\par
   $\Lambda_c(G)$: the coroot lattice;\par
   $\Lambda_w(G)$: the weight lattice;\par
   $T(G)$: a maximal torus;\par
   $ad(G)$: the adjoint group of $G$, i.e. $G/C(G)$ where $C(G)$
   is the center of $G$;\par
   $\Delta(G)$: a simple root system of $G$.\par
   $Out(G)$: the outer automorphism group of $G$, which is defined as the quotient of the automorphism group of $G$ by
    its inner automorphism group. It is well-known that $Out(G)$ is isomorphic to the diagram
   automorphism group of the Dynkin diagram of $G$.\par
      When there is no confusion, we just ignore  the letter $G$.

\section{Reductions to simply laced cases}\label{non-simply->simply}
     Let $G$ be any simple, compact and simply
connected Lie group. Then $G$ is classified into the following $7$
types according to its Lie algebra. \par
   (1) $A_n$-type, $G=SU(n+1)$;\par
   (2) $B_n$-type, $G=Spin(2n+1)$;\par
    (3) $C_n$-type, $G=Sp(n)$;\par
    (4) $D_n$-type, $G=Spin(2n)$;\par
    (5) $E_n$-type, $n=6,7,8$;\par
    (6) $F_4$-type;\par
    (7) $G_2$-type.\par
  Among these, $A_n,D_n$ and $E_n$ are called of simply laced
type, while $B_n,C_n,F_4$ and $G_2$ are called of non-simply laced
type. And $A_n,B_n,C_n,D_n$ are called classical Lie groups, while
  $E_n,F_4$ and $G_2$ are called exceptional Lie groups. \\

   From now on, we always assume that  $G$ is a compact,
simple, simply-connected Lie group of non-simply laced type, that
is, of type $B_n,C_n,F_4,G_2$. There are two natural approaches to reduce situations to simply laced cases.
One is embedding $G$ into a simply laced Lie group $G'$ such that $G$ is
the subgroup fixed by the outer automorphism group of $G'$. Another
 is taking the simply laced subgroup $G''$ of maximal rank. \par
    In the following we explain the first reduction. The following result is
well-known.

\begin{proposition}\label{reduction-I} Let $G$ be a compact, non-simply laced, simple, and simply
connected Lie group. There exists a simple, simply connected and
simply laced Lie group $G'$, s.t. $G\subset G'$ and $G=(G')^{\rho}$,
where $\rho$ is an outer automorphism of $G'$ of order $3$ for
$G'=D_4$, and of order $2$ otherwise.
\end{proposition}
 \begin{proof} By the functorial property, we just need to prove it in the Lie algebra
level. For the construction of $\mathcal{G}=Lie(G)$ and
$\mathcal{G}'=Lie(G')$, one can see \cite{Kac} for the details,
where the construction of Lie algebras is determined by the
construction of root systems. \end{proof}

{\bf Remark} \label{root-sys} {\rm For later use, we list the construction of
non-simply laced root systems via simply laced root systems.}

\begin{enumerate}
\item $G=C_n=Sp(n)$, $G'=A_{2n-1}=SU(2n)$.\\
$\Delta(G')=\{\alpha_i, i=1,\cdots,2n-1\}$. \\ $Out(G')=\{1,\rho
\}\cong \bbZ_2$, where $\rho(\alpha_i)=\alpha_{2n-i},i=1,\cdots,n-1$, $\rho(\alpha_n)=\alpha_n$.\\
$\Delta(G)=\{\beta_i=\frac{1}{2}(\alpha_i+\alpha_{2n-i}),i=1,\cdots,n-1,\beta_n=\alpha_n\}$.\\

\item $G=B_n=Spin(2n+1)$, $G'=D_{n+1}=Spin(2n+2)$.\\
$\Delta(G')=\{\alpha_i, i=1,\cdots,n+1\}$. \\ $Out(G')=\{1,\rho \}\cong\bbZ_2$,
where $\rho(\alpha_i)=\alpha_i,i=3,\cdots,n+1$,
$\rho(\alpha_1)=\alpha_2$, $\rho(\alpha_2)=\alpha_1$.\\
$\Delta(G)=\{\beta_1=\frac{1}{2}(\alpha_1+\alpha_2),\beta_i=\alpha_{i+1},i=2,\cdots,n\}$.\\

\item $G=F_4$, $G'=E_6$.\\ $\Delta(G')=\{\alpha_i, i=1,\cdots,6\}$.\\ $Out(G')=\{1,\rho \}\cong\bbZ_2$,  where
$\rho(\alpha_i)=\alpha_{6-i},i=1,\cdots,5$,  and $\rho(\alpha_6)=\alpha_6$.\\
$\Delta(G)=\{\beta_1=\frac{1}{2}(\alpha_1+\alpha_5),\beta_2=\frac{1}{2}(\alpha_2+\alpha_4),
\beta_3=\alpha_3,\beta_4=\alpha_6\}$.\\

\item $G=G_2$, $G'=D_4=Spin(8)$.\\ $\Delta(G')=\{\alpha_i,
i=1,\cdots,4\}$. \\ $Out(G')=\langle\rho_1,\rho_2 \rangle\cong S_3$,  where $\rho_1$
interchanges $\alpha_1$ and $\alpha_2$, and $\rho_2$ interchanges
$\alpha_1$ and $\alpha_4$.\\
$\Delta(G)=\{\beta_1=\frac{1}{3}(\alpha_1+\alpha_2+\alpha_4),\beta_2=\alpha_3\}$.\\
\end{enumerate}
  The Dynkin diagrams of $G$ and $G'$ are listed in Figure 9.

\begin{figure}[H]
\centering
\includegraphics{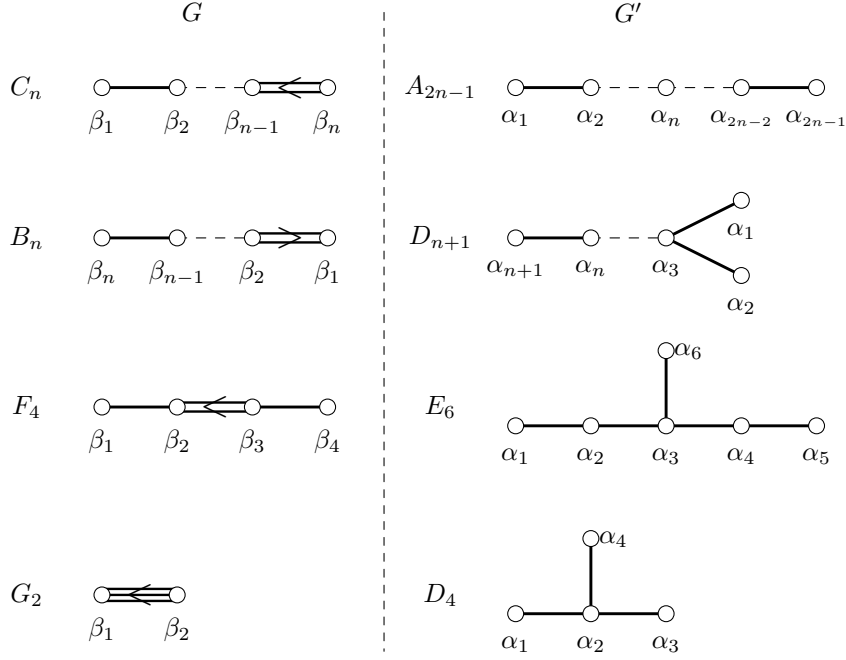}
\caption{Non-simply laced $G$ reduced to simply laced $G'$}
\end{figure}

{\bf Remark} \label{Weyl-group} {\rm  Note that $W(G)$ is the subgroup of $W(G')$ fixing the root
system
  $R(G)$, and also the subgroup  pointwise fixed by $Out(G')$. For a root $\alpha$,
  let $S_{\alpha}\in W(G)$ be the reflection with respect to $\alpha$, that is, $S_{\alpha}(x)=x+(x,\alpha)\alpha$.
  Thus as a subgroup of $W(A_{2n-1})$,
$W(C_n)$ is generated by $S_{\alpha_i}\circ S_{\alpha_{2n-i}}$ for
$i=1,\cdots,n-1$ and $S_{\alpha_n}$. As a subgroup of $W(D_{n+1})$,
$W(B_n)$ is generated by $S_{\alpha_1}\circ S_{\alpha_2}$ and
$S_{\alpha_i}$ for $i=3,\cdots,n+1$. As a subgroup of $W(E_{6})$,
$W(F_4)$ is generated by $S_{\alpha_1}\circ S_{\alpha_5}$,
$S_{\alpha_2}\circ S_{\alpha_4}$, $S_{\alpha_3}$ and $S_{\alpha_6}$.
As a subgroup of $W(D_{4})$, $W(G_2)$ is generated by
$S_{\alpha_1}\circ S_{\alpha_2}\circ S_{\alpha_4}$ and
$S_{\alpha_3}$.}

  In the following we let $\Sigma$ be a fixed elliptic curve with identity element
$0$, and we fix a primitive $d^{th}$ root of $Jac(\Sigma) \cong \Sigma$, where $d=2$ for $D_n$
case, $d=9-n$ for $E_n$ case,  and $d=n+1$ for $A_n$ case, respectively (see \cite{LZ}).
Recall that for any compact, simple and simply-connected Lie group
$H$, the moduli space of flat $H$ bundles over $\Sigma$ is
$$\mathcal{M}_{\Sigma}^{H}\cong (\Lambda_c(H)\otimes \Sigma)/W(H).$$

\vspace{0.3cm}
  For $G'$, the group $Out(G')$ acts on
  $$(\Lambda_c(G')\otimes \Sigma)/W(G')$$ naturally. Let $((\Lambda_c(G')\otimes
\Sigma)/W(G'))^{Out(G')}$ be the fixed part. \par
  Then we have a natural map $$\chi:(\Lambda_c(G)\otimes \Sigma)/W(G)\rightarrow ((\Lambda_c(G')\otimes
\Sigma)/W(G'))^{Out(G')}.$$ The image of $\chi$ is contained in a connected
component of the fixed part.

\begin{lemma}\label{inv-component} The map $$\chi: (\Lambda_c(G)\otimes
\Sigma)/W(G)\rightarrow ((\Lambda_c(G')\otimes
\Sigma)/W(G'))^{Out(G')}$$ is injective. \end{lemma}

\begin{proof} It suffices to prove that for any $x,y\in \Lambda(G)\otimes
\Sigma$, if $\exists\ w'\in W(G')$, such that $w'(x)=y$, then
$\exists\ w\in W(G)$, such that $w(x)=y$. For
 $A_n$ and $D_n$ cases, this is obvious if we check the root lattices.
 For $E_6$ case, we can also check it directly with the help of
 computer. Of course we can also check this case by hand following the discussion
 in Section~\ref{E6}.    \end{proof}

\begin{corollary}\label{fixed-part} (i) The fixed part
$((\Lambda_c(G')\otimes \Sigma)/W(G'))^{Out(G')}$  is determined by
the condition $\rho (x)=x$, up to  $W(G')$-action, where $x\in
\Lambda_c(G')\otimes \Sigma$, and $\rho$ is a generator of
$Out(G')$, of order $3$ for $G'=D_4$ and of order $2$ for $G'=A_n,\
E_n$.
\par (ii) The moduli space $\mathcal{M}_{\Sigma}^{G}\cong
(\Lambda_c(G)\otimes \Sigma)/W(G)$ is a connected component  of the
fixed part
$$(\mathcal{M}_{\Sigma}^{G'})^{Out(G')}\cong((\Lambda_c(G')\otimes
\Sigma)/W(G'))^{Out(G')}$$ containing the trivial $G'$ bundle.
\end{corollary}

\begin{proof}  (i) For any $x\in
\Lambda_c(G')\otimes \Sigma$, denote $\bar{x}$ the class in
$(\Lambda_c(G')\otimes \Sigma)/W(G')$. Then $\rho (\bar{x})=\bar{x}$
if and only if there exists $w\in W(G')$, such that $\rho (x)=w(x)$.
Thus $w^{-1}\rho(x)=x$. But $w^{-1}\rho\in Out(G')$ since
$Out(G')=Aut(G')/W(G')$. Thus we can take a new simple root system
such that $w^{-1}\rho$ is the generator of the diagram automorphism
(the automorphism of order $3$ for $D_4$).
\par (ii) By (i), $(\Lambda_c(G)\otimes \Sigma)/W(G)$ and
$(\mathcal{M}_{\Sigma}^{G'})^{Out(G')}$ are both orbifolds with the same dimension.
Thus the result follows from
Lemma~\ref{inv-component}.  \end{proof}

   If we express the moduli space of flat $G$ bundles over $\Sigma$ as
$(T\times T)/W(G)$, where $T$ is a maximal torus of $G$, then we
have the following corollary.

\begin{corollary} If two elements of $T\times
T$ are conjugate under $W(G')$, then they are also conjugate under
$W(G)$. \end{corollary}

   Another method is to reduce $G$ to its simply-laced subgroup $G''$ of maximal rank, and apply the results
   for simply laced cases to current situation. In another occasion we will discuss our moduli space of
   $G$-bundles from this aspect in detail. Here we just mention the following well-known fact from Lie theory.
\begin{proposition}\label{reduct-II} There exists canonically a simply laced Lie
subgroup $G''$ of $G$, which is of maximal rank, that is, $G''$ and
$G$ share a common maximal torus.  And there is a short exact
sequence $$1\rightarrow W(G'')\rightarrow W(G)\rightarrow Out(G'')
\rightarrow 1,$$ where $Out(G'')$ is the outer automorphism group of
$G''$. Thus, if we write the moduli space as
$\mathcal{M}_{\Sigma}^{G}=(T\times T)/W$, then
$$\mathcal{M}_{\Sigma}^{G}=\mathcal{M}_{\Sigma}^{G''}/Out(G'').\eqno\Box$$
\end{proposition}

{\bf Remark} \label{G-and-G''} {\rm We give this construction of $G''$ in each case.\par
  (1) For $G=Sp(n),\ G''=SU(2)^n$. $Out(G'')$ is the group $S_n$ of
  permutations of the $n$ copies of $SU(2)$ in $G''$.\par
  (2) For $G=G_2,\ G''=SU(3)$. $Out(G'')$ is the group $\mathbb{Z}_2$
  that exchanges the $3$-dimensional representation of $SU(3)$ with
  its dual.\par
  (3) For $G=Spin(2n+1),\ G''=Spin(2n)$. $Out(G'')$ is the group $\mathbb{Z}_2$
  that exchanges the two spin representations of $Spin(2n)$.\par
  (4) For $G=F_4,\ G''=Spin(8)$. $Out(G'')$ is the triality group $S_3$
  that permutes the three $8$-dimensional representations of
  $Spin(8)$.}

\section{Flat $G$ bundles over elliptic curves and rational surfaces: non-simply laced
cases}\label{surface}

   In this section, we study case by case the $G$ bundles over elliptic
curves and rational surfaces for a non-simply laced Lie group $G$.

\subsection{The $B_n(n\geq 2)$ bundles}\label{Bn-surface}

     According to the  arguments of last section, for
  $G=Spin(2n+1)$  we can take $G'=Spin(2n+2)$, such that $G=(G')^{Out(G')}$. \par

     Let $S=Y_{n+1}$ be  a  rational surface with a $D_{n+1}$-configuration \cite{LZ} which contains $\Sigma$
  as a smooth anti-canonical curve. Recall \cite{LZ} that $Y_{n+1}$ is a
   blow-up of $\mathbb{F}_1$ at $n+1$ points $x_1,\cdots,x_{n+1}$ on $\Sigma$, with corresponding exceptional classes
   $l_1,\cdots,l_{n+1}$. Let $f$ be the class of fibers in
   $\bbF_1$, and  $s$ be the section such that  $0=s\cap \Sigma$ is the identity element of
   $\Sigma$.  The Picard group of $Y_{n+1}$ is $H^2(Y_{n+1},
   \mathbb{Z})$, which  is a lattice with basis
   $s,f,l_1,\cdots,l_{n+1}$. The canonical line bundle $K=-(2s+3f-\sum\limits_{i=1}^{n+1}l_i)$.\par

      We know from \cite{LZ} that $$P_{n+1} := \{x\in H^2(Y_{n+1}, \mathbb{Z})\ |\  x\cdot K=x\cdot
      f=0\}$$ is a root lattice of $D_{n+1}$ type. We take a simple root system of $G'$ as
$$\Delta(D_{n+1})=\{\alpha_1=l_1-l_2,\alpha_2=f-l_1-l_2,\alpha_3=l_2-l_3,\cdots,\alpha_{n+1}=l_n-l_{n+1}\}. $$ Let
$\rho$ be the generator of $Out(G')\cong\bbZ_2$, such that
$\rho(\alpha_1)=\alpha_2, \rho(\alpha_2)=\alpha_1$ and
$\rho(\alpha_i)=\alpha_i$ for $i=3,\cdots,n+1$.

   From \cite{LZ} we know that the pair $(S,\Sigma)$ determines a homomorphism $$u\in
   Hom(\Lambda(G'),\Sigma)$$
   which is given by the restriction map: $$u(\alpha)=\calO(\alpha)|_{\Sigma}.$$

\begin{lemma}\label{Bn-u} Let $u\in Hom(\Lambda(G'),\Sigma)$ correspond to a pair
$(S,\Sigma)$, where $S$ is a surface with a $D_{n+1}$-configuration. Then $\rho\cdot u=u$ if and only if  $2x_1=0$.
\end{lemma}

\begin{proof} Since  $u$ is the restriction map: $\alpha_i\mapsto
\calO(\alpha_i)|_{\Sigma}$,
$u(\alpha_1)=\calO(l_1-l_2)|_{\Sigma}=x_1-x_2$, and
$u(\alpha_2)=\calO(f-l_1-l_2)|_{\Sigma}=-x_1-x_2$. Hence $\rho \cdot
u=u\Leftrightarrow u(\alpha_1)=u(\alpha_2) \Leftrightarrow
x_1-x_2=-x_1-x_2\Leftrightarrow 2x_1=0\Leftrightarrow x_1$ is one of
the $4$ points of order $2$ on the elliptic curve $\Sigma$.\end{proof}

   As in \cite{LZ}, we denote $\calS(\Sigma,G')$ the moduli space of $G'=D_{n+1}$-surfaces
with a fixed anti-canonical curve $\Sigma$, and
$\overline{\calS(\Sigma,G')}$ the natural compactification by including
all rational surfaces with
$D_{n+1}$-configurations (Figure 1). From  \cite{LZ} we know that
$\phi:\overline{\calS(\Sigma,G')}\isomap\mathcal {M}_{\Sigma}^{G'}$ is an isomorphism.

\begin{corollary}\label{Bn-bundle}  For
$u\in\mathcal{M}_{\Sigma}^{G}\hookrightarrow(\mathcal{M}_{\Sigma}^{G'})^{Out(G')}$,\
$\phi^{-1}(u)\in\overline{\mathcal {S}(\Sigma,G')}$ represents a
class of surfaces $Y_{n+1}(x_1,\cdots,x_{n+1})$ with $x_1=0$, and
such a surface corresponds to a boundary point in the moduli space,
that is, $\phi^{-1}(u)\in \overline{\mathcal
{S}(\Sigma,G')}\backslash\calS(\Sigma,G')$.
\end{corollary}
 \begin{proof} By Lemma~\ref{Bn-u}, $u\in(\mathcal{M}_{\Sigma}^{G'})^{Out(G')}$ if and only if
 $2x_1=0$. There are $4$ connected components corresponding to $4$
 points of order $2$ on $\Sigma$. Since $\mathcal{M}_{\Sigma}^{G}$
 is the component containing the trivial $G'$ bundle, we have
 $x_1=0$. Recall ($\S4$, \cite{LZ}) that $Y_{n+1}(x_1,\cdots,x_{n+1})\in \calS(\Sigma,G')$
 if and only if $0,x_1,\cdots,x_{n+1}$ are in general position, which implies in particular $x_1\neq
 0$. Hence $\phi^{-1}(u)$ corresponds to a boundary point. \end{proof}

   Denote $S=Y_{n+1}'(x_1=0,x_2,\cdots,x_{n+1})$ (or $Y_{n+1}'$ for brevity) the blow-up of $\mathbb{F}_1$ at
$n+1$ points $x_1=0,x_2,\cdots,x_{n+1}$ on $\Sigma$, with
exceptional divisors $l_1,l_2,\cdots,l_{n+1}$, where
$\Sigma\in|-K_S|$. Similar to the simply laced cases, we give the
following definition.

\begin{definition}\label{Bn-config}   {\rm A  {\it $B_n$-exceptional system} on $S$ is an
$n$-tuple $(e_1,e_2,\cdots,e_{n+1})$ where $e_i$'s are exceptional
divisors such that $e_i\cdot e_j=0=e_i\cdot f, i\neq
j$ and $y_1=e_1\cap\Sigma=0$ is the identity of $\Sigma$. A $B_n$-{\it configuration} on $S$ is a  $B_n$-exceptional
system $\zeta_{B_n}=(e_1,e_2,\cdots,e_{n+1})$ such that we can consider
$S$ as a blow-up of $\bbF_1$ at  $n+1$ points
$y_1=0,y_2,\cdots,y_{n+1}$ on $\Sigma$, that is
$S=Y_{n+1}'(y_1=0,y_2,\cdots,y_{n+1})$, with corresponding exceptional
divisors $e_1,e_2,\cdots$, $e_{n+1}$.
When $S$ has a $B_n$-configuration, we call $S$ a {\it (rational)
surface with a $B_n$-configuration} (see Figure 2). }\end{definition}

   When $x_2,\cdots,x_{n+1}\in \Sigma$ with $x_i\neq 0$ for all $i$ are in general position
(refer to $\S4$ of \cite{LZ} for definition), any $B_n$-exceptional
system on $S$ consists of exceptional curves. Such a surface is
called a $B_n$-{\it surface}. So a $B_n$-surface must have a
$B_n$-configuration.

\begin{lemma} \label{W(Bn)action} (i) Let $S$ be a rational surface with a $B_n$-configuration.
Then the Weyl group $W(B_n)$ acts on all $B_n$-exceptional systems
on $S$ simply transitively.\par
   (ii) Let $S$ be a $B_n$-surface. Then the Weyl group $W(B_n)$ acts on
   all $B_n$-configurations simply transitively.\end{lemma}

\begin{proof} It suffices to prove (i). Let $(e_1,e_2,\cdots,e_{n+1})$ be a
$B_n$-exceptional system on $S$. By Definition~\ref{Bn-config},
$e_i=l_{\sigma(i)}$ or $f-l_{\sigma(i)}$ for $i\neq 1$, where $\sigma$ is a
permutation of $\{2,\cdots,n+1\}$. Note that the Weyl group $W(B_n)$ acts as the group
generated by permutations of the $n$ pairs $\{(l_i,f-l_i)\ |\
i=2,\cdots,n+1\}$ and interchanging of $l_i$ and $f-l_i$ in each
pair $(l_i,f-l_i)_{i\geq 2}$. Then the result follows. \end{proof}

   Let $\mathcal {S}(\Sigma,B_n)$ be the moduli space of pairs $(S,\Sigma)$
where $S$ is a $B_n$-surface (so the blown-up points
$x_1=0,x_2,\cdots,x_{n+1}$ are in general position), and $\Sigma\in
|-K_S|$. Denote $\mathcal{M}_{\Sigma}^{B_n}$ the moduli space of
flat $B_n$ bundles over $\Sigma$. Then applying
Corollary~\ref{Bn-bundle} we have the following identification.

\begin{proposition}\label{Bn-moduli} (i) $\mathcal {S}(\Sigma,B_n)$ is
embedded into $\mathcal{M}_{\Sigma}^{B_n}$ as an open dense subset.
\par (ii) Moreover, this embedding can be extended naturally to an
isomorphism $$\overline{\mathcal {S}(\Sigma,B_n)}\cong
\mathcal{M}_{\Sigma}^{B_n},$$ by including all rational
surfaces with $B_n$-configurations. \end{proposition}

\begin{proof} The proof is similar to that in $ADE$ cases \cite{LZ}.
Firstly, we have $\mathcal{M}_{\Sigma}^{B_n}\cong
\Lambda_c(B_n)\otimes_{\bbZ}\Sigma/W(B_n)$, and
$\Lambda_c(B_n)\otimes_{\bbZ}\Sigma/W(B_n)\cong
Hom(\Lambda(B_n),\Sigma)/W(B_n)$ when we fixed the square root of
unity of $Jac(\Sigma)\cong\Sigma$. Refer to Section 3 of \cite{LZ}
for the detail. \par
  Secondly, the restriction from $S$ to $\Sigma$ induces a map (again denoted by $\phi$)
  $$\phi:\mathcal {S}(\Sigma,B_n)\rightarrow Hom(\Lambda(B_n),\Sigma)/W(B_n).$$ This map is  well-defined, since
  by Lemma~\ref{W(Bn)action}, choosing and fixing a $B_n$-configuration on $S$
  is equivalent to choosing and fixing a system of simple roots $\Delta(B_n)$.
   \par
  Thirdly, the map $\phi$ is injective. For this, we take a simple
  root system of $B_n$ as
$$\beta_1=f-2\ l_2\mbox{ and } \beta_k=2\ \alpha_{k+1}\mbox{ for } 2\leq k\leq n.$$ Then
the restriction induces an element $u\in Hom(\Lambda(B_n),\Sigma)$,
which satisfies the following system of linear equations

$$\left \{ \begin{array}{l} -2\ x_2=p_1,\\
      2(x_k-x_{k+1})=p_k,\ k=2,\cdots,n.
      \end{array}
\right.$$ where $p_i=u(\beta_i)$. Obviously, the solution of this
system of linear equations  exists uniquely for given $p_i$ with $1\leq
i\leq n$. \par
   Finally, the statement (ii) comes from Corollary~\ref{Bn-bundle} and the existence of the solutions to the
above system of linear equations.\end{proof}

{\bf Remark} \label{Torelli}{\rm The situation here is very similar to that in
the compactification theory of the moduli space of (projective) $K3$
surfaces. A natural question there is how to extend the global
Torelli theorem to the boundary components of a compactification
(\cite{Kulikov}\cite{Persson-Pinkham} etc). If we consider
the map $\phi: \calS(\Sigma,G)\rightarrow \calM_{\Sigma}^G$
\cite{LZ} for $G=A_n,D_n$ or $E_n$ as a type of period map, then the
main result of \cite{LZ} is a type of global Torelli theorem. And
Proposition~\ref{Bn-moduli} implies that we can extend the theorem
of Torelli type in $D_{n+1}$ case to a boundary component of the
natural compactification.}

   In the following, we let $S=Y_{n+1}(x_1,\cdots,x_{n+1})$ be the blow-up of $\bbF_1$ at $n+1$ points.
We can construct a Lie algebra bundle on $S$. Here we don't need the
existence of the anti-canonical curve $\Sigma$. According to
Section~\ref{non-simply->simply}, we have a root system of $B_n$
type consisting of divisors on $S$:
$$R(B_n)\triangleq\{\pm(f-2\ l_i),2(l_i-l_j),\pm 2(f-l_i-l_j)\ \ |\
\ i\neq j, 2\leq i,j\leq n+1 \}.$$ Thus we can construct a Lie
algebra bundle of $B_n$-type over $S$:
     $$\mathscr{B}_n\triangleq\mathcal {O}^{\bigoplus n}\bigoplus_{D\in R(B_n)}\mathcal{O}(D).$$
     The fiberwise Lie algebra structure of $\mathscr{B}_n$ is defined as
     follows (the argument here is the same as that in \cite{LZ}). \par
     Fix the system of simple roots of $R_n$ as
     $$\Delta(B_n)=\{\alpha_1=f-2l_2,\alpha_2=2(l_2-l_3),\cdots,\alpha_n=2(l_n-l_{n+1})\},$$ and take a trivialization
   of $\mathscr{B}_n$. Then over a trivializing open subset $U$,
   $\mathscr{B}_n|_U\cong U\times(\mathbb{C}^{\oplus n}\bigoplus_{\alpha\in
   R_n} \mathbb{C}_{\alpha})$. Take a Chevalley basis $\{x_{\alpha}^U,\alpha\in R_n;h_i,1\leq i\leq n\}$
   for $\mathscr{B}_n|_U$  and define the Lie algebra structure by the  following
   four relations, namely, Serre's relations on Chevalley basis (see \cite{Humph1},
   p147):\\

 (a) $[h_ih_j]=0,1\leq i, j\leq n.$\par
 (b) $[h_ix_{\alpha}^U]=\langle\alpha,\alpha_i\rangle x_{\alpha}^U,1\leq i\leq n,\alpha\in
     R_n.$\par
 (c) $[x_{\alpha}^Ux_{-\alpha}^U]=h_{\alpha}$  is a $\mathbb{Z}$-linearly combination of
     $h_1,\cdots,h_n.$\par
 (d) If $\alpha,\beta$ are independent roots, and
     $\beta-r\alpha,\cdots,\beta+q\alpha$ are the
     $\alpha$-string through $\beta$,
     then $[x_{\alpha}^Ux_{\beta}^U]=0$ if $q=0$, while $[x_{\alpha}^Ux_{\beta}^U]=\pm(r+1)x_{\alpha+\beta}^U$ if
     $\alpha+\beta\in R_n.$\\

   Note that $h_i,1\leq i\leq n$ are independent of any trivialization,
so the relation $(a)$ is always invariant under different
trivializations. If $\mathscr{B}_n|_V\cong
V\times(\mathbb{C}^{\oplus n}\bigoplus_{\alpha\in R_n})$ is another
trivialization, and $f_{\alpha}^{UV}$ is the transition function for
the line bundle $\mathcal{O}(\alpha)(\alpha\in R_n)$, that is,
$x_{\alpha}^U=f_{\alpha}^{UV}x_{\alpha}^V$, then the relation $(b)$
is
  $$[h_i(f_{\alpha}^{UV}x_{\alpha}^V)]=\langle\alpha,\alpha_i\rangle f_{\alpha}^{UV}x_{\alpha}^V,$$
   that is, $$[h_ix_{\alpha}^V]=\langle\alpha,\alpha_i\rangle x_{\alpha}^V.$$ So $(b)$
is also invariant. $(c)$ is also invariant since
$(f_{\alpha}^{UV})^{-1}$ is the transition function for
$\mathcal{O}(-\alpha)(\alpha\in R_n)$. Finally, $(d)$ is invariant
since $f_{\alpha}^{UV}f_{\beta}^{UV}$ is the transition function for
$\mathcal{O}(\alpha+\beta)(\alpha,\beta\in R_n)$.\par Therefore, the
Lie algebra structure is compatible with the trivialization. Hence
it is well-defined. \\

    When the surface $S$ contains $\Sigma$ as an anti-canonical curve, restricting the above bundle to
this anti-canonical curve $\Sigma$, we obtain a Lie algebra bundle
of $B_n$-type over $\Sigma$, which determines uniquely a flat $B_n$
bundle over $\Sigma$.  On the other hand, when $x_1=0$, we can
identify these two line bundles $\mathcal{O}_{\Sigma}(l_1)$ and
$\mathcal{O}_{\Sigma}(f-l_1)$ when restricting them to $\Sigma$.
Recall  the spinor bundles $S^+_{n+1}$ and $S^-_{n+1}$ of $D_{n+1}$ are defined as
follows\cite{Leung}\cite{LZ} (here we omit the subscription $n+1$ for brevity)
\begin{eqnarray*}
\mathcal{S}^{+}&=&\bigoplus_{D^2=D\cdot K=-1, D\cdot f=1}\mathcal{O}(D)\mbox{ and }\\
\mathcal{S}^{-}&=&\bigoplus_{T^2=-2,T\cdot K=0, T\cdot f=1}\mathcal{O}(T).
\end{eqnarray*}

The identification of
$\calO_{\Sigma}(l_1)\cong\calO_{\Sigma}(f-l_1)$ induces an
identification of these two spinor bundles $S^+$ and $S^-$, which is
given by (of course, when restricted to $\Sigma$)
$$S^+\otimes \mathcal{O}(-l_1)\cong S^-.$$
   From representation theory, we know this determines a flat $B_n$ bundle over $\Sigma$.\par
   Conversely, if $\mathcal{S}^{+}|_{\Sigma}\cong\mathcal{S}^{-}|_{\Sigma}$, then we must have $x_1=0$ (up to renumbering).
   For example, we consider the $n=2$ case. Note that
\begin{eqnarray}
\mathcal{S}^{+}|_{\Sigma}&\otimes&\calO(-(0))=\mathcal{O}\oplus\mathcal{O}((-x_1-x_2)-(0))\oplus\mathcal{O}((-x_1-x_3)-(0)) \nonumber\\
                                          &&\oplus\mathcal{O}((-x_2-x_3)-(0)),\nonumber\\
\mathcal{S}^{-}|_{\Sigma}&=&\mathcal{O}((0)-(x_1))\oplus\mathcal{O}((0)-(x_2))\oplus\mathcal{O}((0)-(x_3))\nonumber\\
                                          &&\oplus\mathcal{O}(3(0)-(x_1)-(x_2)-(x_3)).\nonumber
\end{eqnarray}
   Where for a point $x\in\Sigma$, $(x)$ means the divisor of degree one,
   and $\calO((x))$ means the line bundle determined by this divisor.
   Thus, $\mathcal{S}^{+}_{\Sigma}\otimes\calO(-(0))=\mathcal{S}^{-}_{\Sigma}$ implies that $x_1=0$ (up to renumbering).
   The general case follows from similar arguments.

\subsection{The $C_n$ bundles}\label{Cn-surface}

    We take $G=C_n\subset G'=A_{2n-1}$,  where $C_n=Sp(n)$   and
    $A_{2n-1}=SU(2n)$. They satisfy the relation $G=(G')^{Out(G')}$.

    Let $S=Z_{2n}$ be  a  rational surface with an
    $A_{2n-1}$-configuration
(see \cite{LZ} or Figure 3) which contains $\Sigma$
  as a smooth anti-canonical curve. Recall \cite{LZ} that $Z_{2n}$ is
  a (successive) blow-up of $\mathbb{F}_1$ at $2n$ points $x_1,\cdots,x_{2n}$ on $\Sigma$,
  with corresponding exceptional classes $\l_1,\cdots,l_{2n}$. Let $f$ be the class of fibers in
   $\bbF_1$, and  $s$ be the section such that  $0=s\cap \Sigma$ is the identity element of
   $\Sigma$.  The Picard group of $Z_{2n}$ is $H^2(Z_{2n},\mathbb{Z})$, which is a lattice with basis
   $s,f,l_1,\cdots,l_{2n}$. The canonical line bundle $K=-(2s+3f-\sum\limits_{i=1}^{2n}l_i)$.\par

   Recall $$P_{2n-1}:=\{x\in H^2(Z_{2n},\mathbb{Z})\ |\ x\cdot K=x\cdot f=x\cdot s=0\} $$
   is a root lattice of $A_{2n-1}$ type. And we can take a simple root system of $A_{2n-1}$ as
$$\Delta(A_{2n-1})=\{\alpha_i=l_i-l_{i+1}\ |\ 1\leq i\leq 2n-1\}.$$ Note that \cite{LZ} we have used
the convention that $\sum\limits_{i=1}^{2n} x_i=0$.

      Let $\rho$ be the generator of $Out(G')\cong\bbZ_2$, such that
$\rho(\alpha_i)=\alpha_{2n-i}$ for $i=1,\cdots,2n-1$.

   When the above simple root system is chosen, the pair $(S,\Sigma)$ determines a homomorphism
   $u\in Hom(\Lambda(G'),\Sigma)$
   which is given by the restriction map $$u(\alpha)=\calO(\alpha)|_{\Sigma}.$$

\begin{lemma}\label{Cn-u} Let $u\in Hom(\Lambda(G'),\Sigma)$ be an element corresponding to a pair $(S,\Sigma)$,
where $S$ is a surface with an $A_{2n-1}$-configuration. Then $\rho\cdot u=u$ if and only if
$n(x_i+x_{2n+1-i})=0$ for $i=1,\cdots,n$.
\end{lemma}

\begin{proof} Since  $u$ is the restriction map: $\alpha_i\mapsto
\calO(\alpha_i)|_{\Sigma}$,
$u(\alpha_i)=\calO(l_i-l_{i+1})|_{\Sigma}=x_i-x_{i+1}$ for
$i=1,\cdots,2n-1$. Hence $\rho \cdot u=u\Leftrightarrow
u(\alpha_i)=u(\alpha_{2n-i}) \Leftrightarrow
x_i-x_{i+1}=x_{2n-i}-x_{2n-i+1}\Leftrightarrow
n(x_i+x_{2n-i+1})=0$ since $\sum\limits_{i=1}^{2n} x_i=0$. \end{proof}

   As in \cite{LZ}, we denote $\calS(\Sigma,G')$ the moduli space of $G'=A_{2n-1}$-surfaces
with a fixed anti-canonical curve $\Sigma$, and
$\overline{\calS(\Sigma,G')}$ the natural compactification by including
all rational surfaces with
$A_{2n-1}$-configurations. From  \cite{LZ} we know that there is an isomorphism
$\phi:\overline{\calS(\Sigma,G')}\isomap\mathcal {M}_{\Sigma}^{G'}$.

\begin{corollary}\label{Cn-bundle}  For
$u\in\mathcal{M}_{\Sigma}^{G}\hookrightarrow(\mathcal{M}_{\Sigma}^{G'})^{Out(G')}$,\
$\phi^{-1}(u)\in\overline{\mathcal {S}(\Sigma,G')}$ represents a
class of surfaces $Z_{2n}(x_1,\cdots,x_{2n})$ with
$x_i+x_{2n+1-i}=0$ for $i=1,\cdots,n$, and such a surface
corresponds to a boundary point in the moduli space, that is,
$\phi^{-1}(u)\in \overline{\mathcal
{S}(\Sigma,G')}\backslash\calS(\Sigma,G')$.
\end{corollary}

 \begin{proof} By Lemma~\ref{Cn-u}, $u\in(\mathcal{M}_{\Sigma}^{G'})^{Out(G')}$ if and only if
 $n(x_i+x_{2n+1-i})=0$ for $i=1,\cdots,n$. There are $n^2$ connected components corresponding to $n^2$
 points of order $n$ on $\Sigma$. Since $\mathcal{M}_{\Sigma}^{G}$
 is the component containing the trivial $G'$ bundle, we have
 $x_i+x_{2n+1-i}=0$ for $i=1,\cdots,n$. Recall ($\S4$, \cite{LZ}) that $Z_{2n}(x_1,\cdots,x_{2n})\in \calS(\Sigma,G')$
 if and only if $0,x_1,\cdots,x_{2n}$ are in general position, which implies in particular $x_i\neq
 -x_{2n+1-i}$. Hence $\phi^{-1}(u)$ corresponds to a boundary point.
 \end{proof}

   Denote $S=Z_{2n}'(\pm x_1,\cdots,\pm x_n)$  the blow-up of $\mathbb{F}_1$ at
$n$ pairs of points $(x_1,-x_1)$, $\cdots$, $(x_n,-x_n)$ on
$\Sigma$, with $n$ pairs of corresponding exceptional divisors
$(l_1,l_1^-)$, $\cdots$, $(l_n,l_n^-)$, where $l_i$ (resp. $l_i^-$)
is the exceptional divisor corresponding to the blowing up at $x_i$
(resp. $-x_i$). Similar to the other cases, we give the following definitions.

\begin{definition}\label{Cn-config}  {\rm A  {\it $C_n$-exceptional system} on $S$ is an
$n$-tuple of pairs $$((e_1,e_1^-),\cdots,(e_n,e_n^-))$$ where
$(e_i,e_i^-)=(l_{\sigma(i)},l_{\sigma(i)}^-)$ or
$(l_{\sigma(i)}^-,l_{\sigma(i)})$, $i=1,\cdots,n$, with $\sigma$ is
a permutation of $1,\cdots,n$. A $C_n$-{\it configuration} on $S$ is
a  $C_n$-exceptional system
$\zeta_{C_n}=((e_1,e_1^-),\cdots,(e_n,e_n^-))$ such that we can blow
down successively $e_1^-$,$\cdots$,$e_n^-$, $e_n$,$\cdots$,$e_1$
such that the resulting surface is  $\bbF_1$ (see Figure
4)}.\end{definition}

   We say that $x_1, x_2,\cdots,x_n\in\Sigma\subset \bbF_1$ are $n$ points {\it in general
   position}, if they satisfy \par
   (i) they are distinct points, and \par
   (ii) for any $i,j$, $x_i+ x_j\neq 0$.\par

   Equivalently, $x_1, x_2,\cdots,x_n\in\Sigma\subset \bbF_1$ are in general
   position if and only if any $C_n$-exceptional system on $S=Z_{2n}'(\pm x_1,\cdots,\pm x_n)$ consists of
smooth exceptional curves. Such a surface is called a {\it
$C_n$-surface}. Thus a $C_n$-surface must have a
$C_n$-configuration.

\begin{lemma}\label{W(Cn)action} (i) Let $S$ be a surface with a $C_n$-configuration. Then the Weyl group $W(C_n)$
acts on all $C_n$-exceptional systems on $S$ simply
 transitively.\par
 (ii) Let $S$ be a $C_n$-surface. Then the Weyl group $W(C_n)$
acts on all $C_n$- configurations on $S$ simply
 transitively.\end{lemma}

 \begin{proof} It suffices to prove (i). The Weyl group $W(C_n)$
 acts as the group generated by permutations of
the $n$ pairs $\{(l_i,l_i^-)\ |\ i=1,\cdots,n\}$ and interchanging
of $l_i$ and $l_i^-$ for each $i$. From this, we see that $W(C_n)$
acts on all $G$-configurations simply
transitively. \end{proof}

    Denote $\mathcal
{S}(\Sigma,C_n)$ the moduli space of pairs $(Z_{2n}',\Sigma)$, where
$Z_{2n}'$ is a $C_n$-surface, that is, the blow-up of $\bbF_1$ at
$2n$ points $\pm x_1,\cdots,\pm x_n$ such that $x_1,\cdots,x_n$ are
in general position. Denote $\mathcal{M}_{\Sigma}^{C_n}$ the moduli
space of flat $C_n$ bundles over $\Sigma$. By Corollary~\ref{Cn-bundle} we have the following identification.

\begin{proposition}\label{Cn-moduli} (i)  $\mathcal {S}(\Sigma,C_n)$ is
embedded into $\mathcal{M}_{\Sigma}^{C_n}$ as an open dense subset.
\par (ii) Moreover, this embedding can be extended  naturally to an
isomorphism $$\overline{\mathcal {S}(\Sigma,C_n)}\cong
\mathcal{M}_{\Sigma}^{C_n},$$ by including all rational surfaces with
$C_n$-configurations. \end{proposition}

\begin{proof} The proof is basically the same as that in $B_n$ case. We only
need to replace the corresponding parts by the following two things.
  Firstly, according to Section~\ref{non-simply->simply}, we can
take a simple root system as
$$\Delta(C_n)=\{\beta_k=\varepsilon_k-\varepsilon_{k+1},\ 1\leq k\leq n-1,\
\beta_n=2\varepsilon_n\},$$ where $\varepsilon_k=l_k-l_k^-,\ 1\leq
k\leq n$. \par
   Secondly, the restriction map gives us the following system of
linear equations:
$$\left \{ \begin{array}{l} 4 x_n=p_n,\\
      2(x_k-x_{k+1})=p_k,\ k=1,\cdots,n-1.
      \end{array}
\right.$$  The solution of this system  exists uniquely.\end{proof}

{\bf Remark}  {\rm As in $B_n$ case, the above proposition is also similar to
extending the Torelli theorem to a certain boundary component.}

{\bf Remark}  {\rm Obviously, this description in Proposition~\ref{Cn-moduli} coincides with the
well-known description of flat $C_n$ bundles over elliptic curves
\cite{FMW1}. A flat $C_n=Sp(n)$ bundle over $\Sigma$ corresponds to
$n$ pairs (unordered) of points $(x_i,-x_i),i=1,\cdots,n$ on
$\Sigma$, uniquely up to isomorphism. And one pair $(x_i,-x_i)$ will
determine exactly one point on $\mathbb{CP}^1$, since the rational
map determined by the linear system $| 2(0) | $ induces a
double covering from $\Sigma$ onto $\mathbb{CP}^1$. So the moduli
space of flat $C_n$ bundles over $\Sigma$ is just isomorphic to
$S^n(\mathbb{CP}^1)=\mathbb{CP}^n$, the ordinary projective $n$
space.}

   As in $B_n$ case, we construct a Lie algebra bundle of $C_n$ type over
$Z_{2n}'$:
     $$\mathscr{C}_n=\mathcal {O}^{\bigoplus n}\bigoplus_{D\in R(C_n)}\mathcal{O}(D),$$
where $R(C_n)$ is the root system of $C_n$ according to
Section~\ref{non-simply->simply}:
$$R(C_n)=\{\pm 2(l_i-l_i^-),\pm ((l_i-l_i^-)\pm(l_j-l_j^-) )\ |\ i\neq j, 1\leq i,j\leq n \}.$$

      Recall \cite{LZ} the first fundamental representation bundle of $\mathscr{A}_{2n-1}$ is
$$\mathcal{V}_{2n-1}=\bigoplus\limits_{i=1}^{2n}\mathcal{O}(l_i).$$
      The condition that $x_i+x_{2n+1-i}=0,1\leq i\leq n$
is equivalent to an identification of the following two fundamental
representation bundles $\wedge^i(\mathcal{V}_{2n-1})$ and
$\wedge^{2n-i}(\mathcal{V}_{2n-1})$ with $i=1,\cdots,n-1$, which is given by (of course,
when restricted to $\Sigma$)
$$(\wedge^i(\mathcal{V}_{2n-1}))^*\otimes det(\mathcal{V}_{2n-1})\cong \wedge^{2n-i}(\mathcal{V}_{2n-1}).$$
Note that when restricted to $\Sigma$,  the line bundle
$det(\mathcal{V}_{2n-1})=\mathcal{O}(l_1+\cdots l_{2n})$ is
isomorphic to $\mathcal{O}(nf)|_{\Sigma}= \mathcal{O}_{\Sigma}(2n(0))$, by our assumption that $\sum x_i=0$. This identification
determines uniquely a flat $C_n$ bundle over $\Sigma$.

\subsection{The $G_2$ bundles}\label{G2}

    For $G=G_2$, we take $G'=D_4=Spin(8)$ such that $G=(G')^{Out(G')}$. \par
    Let $S=Y_4$ be a rational surface with a $D_4$-configuration
  \cite{LZ} which contains $\Sigma$
  as a smooth anti-canonical curve. Recall (\cite{LZ} or Figure 5) that $Y_4$ is
  a (successive) blow-up of $\mathbb{F}_1$ at $4$ points $x_1,\cdots,x_{4}$ on $\Sigma$, with corresponding
  exceptional classes $l_1,\cdots,l_4$. Let $f$ be the class of fibers in
   $\bbF_1$, and $s$ be the section such that  $0=s\cap \Sigma$ is the identity element of
   $\Sigma$.  The Picard group of $Y_4$ is $H^2(Y_4,\mathbb{Z})$, which is a lattice with basis
   $s,f,l_1,\cdots,l_{4}$. The canonical line bundle $K=-(2s+3f-\sum\limits_{i=1}^{4}l_i)$.\par

   Recall $$P_{4}:=\{x\in H^2(Y_4,\mathbb{Z})\ |\ x\cdot K=x\cdot f=0\} $$
   is a root lattice of $D_4$-type. And we can take a simple root system of $D_4$ as
$$\Delta(D_4)=\{\alpha_1=l_1-l_2,\alpha_2=f-l_1-l_2,\alpha_3=l_2-l_3,\alpha_4=l_3-l_4\}.$$
      Let $\rho\in Out(G')\cong S_3$ (the permutation group of $3$ letters ) be the triality automorphism of order $3$,
such that $\rho(\alpha_1)=\alpha_2$, $\rho(\alpha_2)=\alpha_4$,
$\rho(\alpha_4)=\alpha_1$, and $\rho(\alpha_3)=\alpha_3$.

   When the above simple root system is chosen, the pair $(S,\Sigma)$ determines a homomorphism
   $u\in Hom(\Lambda(G'),\Sigma)$
   which is given by the restriction map $$u(\alpha)=\calO(\alpha)|_{\Sigma}.$$

\begin{lemma}\label{G2-u} Let $u\in Hom(\Lambda(G'),\Sigma)$ correspond to the pair $(S,\Sigma)$,
where $S$ is a surface with a $D_4$-configuration. Then $\rho\cdot u=u$ if and only if
$2x_1=0$ and $x_1+x_4=x_2+x_3$.
\end{lemma}

\begin{proof} Since  $u$ is the restriction map: $\alpha_i\mapsto
\calO(\alpha_i)|_{\Sigma}$,
$u(\alpha_1)=\calO(l_1-l_2)|_{\Sigma}=x_1-x_2$,
$u(\alpha_2)=-x_1-x_2$, $u(\alpha_4)=x_3-x_4$, and
$u(\alpha_3)=x_2-x_3$. Hence $\rho \cdot u=u$ $\Leftrightarrow$ $
u(\alpha_1)=u(\alpha_2)=u(\alpha_4)$  $\Leftrightarrow$ $
x_1-x_2=-x_1-x_2=x_3-x_4$ $\Leftrightarrow$ $
2x_1=0$ and $x_1+x_4=x_2+x_3$.\end{proof}

   Denote $\calS(\Sigma,G')$ the moduli space of $G'=D_4$-surfaces
with a fixed anti-canonical curve $\Sigma$, and
$\overline{\calS(\Sigma,G')}$ the natural compactification by including all rational surfaces with $D_4$-configurations.
From \cite{LZ} we know that
$\overline{\calS(\Sigma,G')}\isomap\mathcal {M}_{\Sigma}^{G'}$. Let
$\phi$ be the isomorphism.

\begin{corollary}\label{G2-bundle}  For
$u\in\mathcal{M}_{\Sigma}^{G}\hookrightarrow(\mathcal{M}_{\Sigma}^{G'})^{Out(G')}$,\
$\phi^{-1}(u)\in\overline{\mathcal {S}(\Sigma,G')}$ represents a
class of surfaces $Y_4(x_1,\cdots,x_4)$ with $x_1=0$ and
$x_4=x_2+x_3$, and such a surface corresponds to a boundary point in
the moduli space, that is, $\phi^{-1}(u)\in \overline{\mathcal
{S}(\Sigma,G')}\backslash\calS(\Sigma,G')$.
\end{corollary}

 \begin{proof} By Lemma~\ref{G2-u}, $u\in(\mathcal{M}_{\Sigma}^{G'})^{Out(G')}$ if and only if
 $2x_1=0$ and $x_1+x_4=x_2+x_3$. There are $4$ connected components corresponding to $4$
 points of order $2$ on $\Sigma$. Since $\mathcal{M}_{\Sigma}^{G}$
 is the component containing the trivial $G'$ bundle, we have
 $x_1=0$  and $x_4=x_2+x_3$. Recall that $Y_4(x_1,\cdots,x_4)\in \calS(\Sigma,G')$
 if and only if $0,x_1,\cdots,x_4$ are in general position, which implies in particular $x_1\neq 0$.
 Hence $\phi^{-1}(u)$ corresponds to a boundary point.
 \end{proof}

   Denote $S=Y_4'(x_1,\cdots,x_4)$  the blow-up of $\mathbb{F}_1$ at
$4$ points $x_1,\cdots,x_4$ on $\Sigma$, with $x_1=0$  and
$x_4=x_2+x_3$. Let $l_1,\cdots,l_4$ be the corresponding exceptional
classes. We give the following definition.

\begin{definition}\label{G2-config}   {\rm A  {\it $G_2$-exceptional system} on $S$ is
an ordered triple  $(e_1,e_2,e_3,e_4)$ of exceptional divisors such that
$e_i\cdot e_j=0=e_i\cdot f,i\neq j$ and $y_1=0,\ y_4=y_2+y_3$
where $y_i=e_i\cdot
 \Sigma$. A $G_2$-{\it configuration} on $S$ is a $G_2$-exceptional
system $\zeta_{G_2}=(e_1,e_2,e_3,e_4)$ such that we can consider $S$ as
a blow-up of $\bbF_1$ at these $4$ points $y_1=0,y_2,y_3,y_4$ on
$\Sigma$, that is $S=Y_4'(y_1=0,y_2,y_3,y_4)$, with corresponding
exceptional divisors $e_1,e_2,e_3,e_4$.
When $S$ has a $G_2$-configuration (of course $\Sigma\in |-K_S|$),
we call $S$ a {\it (rational) surface with a $G_2$-configuration}.
  For $S=Y_4'(x_1,\cdots,x_4)$ with $x_1=0$ and $x_4=x_2+x_3$, when
$x_1,\pm x_2,\pm x_3,\pm x_4$ are distinct points on $\Sigma$, any
$G_2$-exceptional system on $S$ consists of exceptional curves. Such
a surface is called a $G_2$-{\it surface}. So a $G_2$-surface must
have a $G_2$-configuration. These four points
$x_1,x_2,x_3,x_4\in\Sigma$ are said to be {\it in general position}.
}\end{definition}

   A $G_2$-configuration is illustrated in Figure 6.

\begin{lemma}\label{W(G2)action} (i) Let $S=Y_4'(x_1,\cdots,x_4)$ with $x_1=0$ and $x_4=x_2+x_3$ be a surface with a
$G_2$-configuration. Then the Weyl group $W(G_2)$ acts on all
$G_2$-exceptional systems on $S$ simply
 transitively.\par
 (ii) Let $S$ be a $G_2$-surface. Then the Weyl group $W(G_2)$
acts on all $G_2$-config-urations on $S$ simply
 transitively.\end{lemma}
\begin{proof} It suffices to prove (i). By an explicit computation, there
are $12$ $G_2$-configurations: $(l_1,l_2,l_3,l_4)$, $(f-l_1,f-l_2,f-l_3,f-l_4)$,
$(f-l_1,f-l_2,l_4,l_3)$, $(f-l_1,l_4,f-l_2,l_3)$, and so on. The rule is keeping
the relation $x_2+x_3=x_4$ fixed. The Weyl group $W(G_2)$ is the
automorphism group of the sub-root system $A_2$ with simple roots
$\{3(l_2-l_3),3(l_3-(f-l_4))\}$, so $W(G_2)\cong \mathbb{Z}_2
\rtimes W(A_2)=\mathbb{Z}_2 \rtimes S_3$. We can also consider
$W(G_2)$ as the subgroup of $W(D_4)$ generated by two elements
$S_{\alpha_1}S_{\alpha_2}S_{\alpha_4}$ and $S_{\alpha_3}$, where
$S_{\alpha}$ means the reflection with respect to a root $\alpha$ of
$D_4$. Thus we can directly
check that $W(G_2)$ acts on all $G_2$-exceptional systems simply
transitively. \end{proof}

\begin{proposition}\label{G2-moduli} Let $\mathcal {S}(\Sigma,G_2)$ be the moduli space
of pairs $(Y_4',\Sigma)$ where $Y_4'$ is a $G_2$-surface, and
$\mathcal{M}_{\Sigma}^{G_2}$ be the moduli space of flat $G_2$
bundles over $\Sigma$. Then we have \par (i) $\mathcal
{S}(\Sigma,G_2)$ is embedded into $\mathcal{M}_{\Sigma}^{G_2}$ as an
open dense subset. \par
 (ii) Moreover, this embedding
can be extended naturally to an isomorphism $$\overline{\mathcal
{S}(\Sigma,G_2)}\cong \mathcal{M}_{\Sigma}^{G_2},$$ by including all
 rational surfaces with $G_2$-configurations.
\end{proposition}

\begin{proof} We just  note that only the following two things are different
from their counterparts of the proofs in $B_n,C_n$ cases. \par (i) Take
a simple root system of $G_2$ as
$$\Delta(G_2)=\{\beta_1=f-2l_2+l_3-l_4,\ \beta_2=3(l_2-l_3)\}.$$\par
(ii) Then the restriction to $\Sigma$ gives us  the following system
of linear equations:
$$\left \{ \begin{array}{l}  3x_2=-p_1,\\
      3(x_2-x_3)=p_2.
      \end{array}
\right.$$
\end{proof}

  As before, we construct a Lie algebra bundle of $G_2$-type over $S=Y_4'$.
For brevity, denote $\varepsilon_1=l_2,\ \varepsilon_2=l_3,\mbox{
and } \varepsilon_3=f-l_4$. Then
     $$\mathscr{G}_2=\mathcal {O}^{\bigoplus 2}\bigoplus_{D\in R(G_2)}\mathcal{O}(D),$$
where $R(G_2)$ is the root system of $G_2$:
$$R(G_2)=\{\pm 3(\varepsilon_i-\varepsilon_j),\pm (2\varepsilon_i-\varepsilon_j-\varepsilon_k)\ |\ i\neq j\neq k, 1\leq i,j,k\leq 3
\}.$$

   Recall \cite{Leung} the $3$ fundamental representation bundles of rank $8$ of $D_4$ are
defined as:

  $$\left \{ \begin{array}{l}  \mathscr{W}_4=\bigoplus\limits_{C^2=C\cdot K=-1,C\cdot f=0}\calO(C),\\
      \mathcal{S}^{+}_4=\bigoplus\limits_{D^2=D\cdot K=-1, D\cdot
      f=1}\mathcal{O}(D),\\
    \mathcal{S}^{-}_4=\bigoplus\limits_{T^2=-2,T\cdot K=0, T\cdot f=1}\mathcal{O}(T).
      \end{array}
  \right. $$\\

     These conditions $x_1=0,\ x_4=x_2+x_3$ enable us to identify $S^+_4,S^-_4$ and $\mathscr{W}_4$ when restricted to
 $\Sigma$, by
  $$S^+_4\otimes\mathcal{O}(-l_1)\cong S^-_4\mbox{ and }S^+_4\cong \mathscr{W}_4\otimes\mathcal{O}(s).$$
   And these identifications determine uniquely a flat $G_2$ bundle
   over $\Sigma$. Conversely, the identification of these three bundles restricted to $\Sigma$
   implies the conditions $x_1=0$ and $x_4=x_2+x_3$ (up to renumbering). Note that
    \begin{eqnarray}&&\mathscr{W}_4|_{\Sigma}=\bigoplus\calO_{\Sigma}(l_i)\bigoplus\calO_{\Sigma}(f-l_i)=\bigoplus\calO((x_i))\bigoplus\calO((-x_i)),\nonumber\\
    &&\mathcal{S}^{-}_4|_{\Sigma}=\bigoplus\limits_{i}\calO((0)-(x_i))\bigoplus\limits_{j}\calO(3(0)-\sum\limits_{i\neq j} (x_i)),\mbox{ and }\nonumber\\
    &&\mathcal{S}^{+}_4|_{\Sigma}=\calO((0))\bigoplus\limits_{i\neq j}\calO((-x_i-x_j))\bigoplus\calO((-\sum x_i)).\nonumber\end{eqnarray}
     So $\mathscr{W}_4|_{\Sigma}=\mathcal{S}^{-}_4$ implies $x_1=0$, and $\mathscr{W}_4|_{\Sigma}=\mathcal{S}^{+}_4$
     implies $x_4=x_2+x_3$.

\subsection{The $F_4$ bundles}\label{F4}

        First we recall some fundamental facts on $E_6$ root systems and cubic surfaces, which are of
    independent interest.
\subsubsection{The root system of $E_6$, revisited }\label{E6}

 The relation between the root system of $E_6$-type and  smooth cubic
surfaces in $\mathbb{CP}^3$ has been studied for a very long time
\cite{Hart}\cite{Demazure}\cite{Manin}. There are $27$ {\it lines} on such
a cubic surface $S$ (a curve on $S$ is a line if and only if it is
an exceptional curve). And every $E_6$-{\it exceptional system} on $S$
is an ordered $6$-tuples of lines $(e_1,\cdots,e_6)$ which are
pairwise disjoint.  The Weyl group $W(E_6)$ is the symmetry group of
all $E_6$-exceptional systems, that is, $W(E_6)$ acts simply
transitively on the set of all $E_6$-exceptional systems. Now we
consider the unordered $6$-tuple $L=\{e_1,\cdots,e_6\}$. There are
$72$ such $6$-tuples. This corresponds to $36$ {\it Schl$\ddot{a}$fli's
double-sixes} $\{L;L'\}$ \cite{Hart}. In the following we consider a
cubic surface $S$ as the blow-up of $\bbP^2$ at $6$ points
$x_1,\cdots,x_6$ in general position, that is
$S=X_6(x_1,\cdots,x_6)$, with corresponding exceptional curves
$l_1,\cdots,l_6$. Fix a simple root system of $E_6$ as
$$\Delta(E_6)=\{\alpha_1,\cdots,\alpha_6\},$$ where
$\alpha_1=l_1-l_2$, $\alpha_2=l_2-l_3$, $\alpha_3=h-l_1-l_2-l_3$,
and $\alpha_i=l_{i-1}-l_i$, for $i=4,5,6$ \cite{LZ}.

\begin{lemma} One double-six $\{L;L'\}$ corresponds to exactly one positive
root of $E_6$.\end{lemma}

\begin{proof} First take $L_0=\{l_1,\cdots,l_6\}$, then
$L_0'=\{l_1',\cdots,l_6'\}=s_{\alpha_0}(L_0)$ where
$\alpha_0=2h-\sum l_i$ is a positive root and
$l_i'=s_{\alpha_0}(l_i)=2h-\sum_{j\neq i}l_j$. $\{L_0;L_0'\}$ forms
a double-six and $\alpha_0(\succ 0)$ is uniquely determined by
$\{L_0;L_0'\}$, since $W(E_6)$ acts simply and transitively. If
$L=g(L_0)$ with $g\in W(E_6)$, then $\{g(L_0);g(L_0')\}$ is also a
double-six. Let $g(L_0')=S_{\alpha}(g(L_0))$, then
$L_0'=(g^{-1}S_{\alpha}g)(L_0)$. So
$g^{-1}S_{\alpha}g=S_{\alpha_0}$. Then $S_{\alpha}=gS_{\alpha_0}
g^{-1}=S_{g(\alpha_0)}$. This implies $\alpha=\pm g(\alpha_0)$. Take
$\alpha\succ 0$. Now if $\alpha=\alpha_0$, then by a result in page 44 of
\cite{Humph2}, $g\in S_6$, that is, $g$ is a permutation of the six lines $l_i$'s.  Thus
$\{L;L'\}\mbox{
and }\{L_0;L_0'\}$ are the same one.      \end{proof}

{\bf Remark}  {\rm Let $\rho$ be an outer automorphism of $E_6$ of order
$2$, such that $\rho(\alpha_1)=\alpha_6,
\rho(\alpha_2)=\alpha_5\mbox{ and }\rho \mbox{ fixes other simple
roots} $. Consider $F_4$ as the fixed part of $E_6$ by $\rho$. Then
the coroot lattice $\Lambda_c(F_4)$ of $F_4$ is
\begin{eqnarray} \Lambda_c(F_4)&=& \Lambda_c(E_6)^{\rho} \nonumber \\
\quad &=& \Lambda(E_6)^{\rho} \nonumber \\
\quad &=& \{a h+\sum a_i l_i\ |\ a_1+a_6=a_2+a_5=a_3+a_4=-a\} \nonumber \\
\quad &=& \mathbb{Z}\langle h-l_1-l_2-l_3,l_1-l_6,l_2-l_5,l_3-l_4\rangle \nonumber \\
\quad &=& \Lambda(D_4). \nonumber \end{eqnarray}

And the Weyl group of $F_4$ is
\begin{eqnarray} W(F_4)&=&\{w\in
W(E_6)\ |\ w \mbox{ preserves }
\Lambda_c(F_4)= \Lambda(D_4)\}\nonumber\\ &=& Aut(\Lambda(D_4))\nonumber\\
&=& S_3 \rtimes  W(D_4).\nonumber\end{eqnarray} }

{\bf Remark} {\rm If $3$ lines $e_1,e_2,e_3$ pairwise intersect, we say that
they form a {\it triangle}. Denote by $\Delta=\{e_1,e_2,e_3\}$ a
(unordered) triangle, and by $\overrightarrow{\Delta}=(e_1,e_2,e_3)$
an ordered triangle. Every line belongs to $5$ triangles, so there are
$27\cdot 5/3=45$ triangles. And if $\{e_1,e_2,e_3\}$ is a triangle,
then $-K=e_1+e_2+e_3$. $W(E_6)$ acts on all these  $45$ triangles
transitively, and $W(F_4)$ is the isotropy subgroup of the triangle
$\Delta_0=\{h-l_1-l_6,h-l_2-l_5,h-l_3-l_4\}$. Moreover $W(D_4)$ is
the isotropy subgroup of the ordered triangle
$\overrightarrow{\Delta}_0=(h-l_1-l_6,h-l_2-l_5,h-l_3-l_4)$. The
reason is the following: \par
   Let $\Delta=\{e_1,e_2,e_3\}$ and
$\Delta'=\{f_1,f_2,f_3\}$ be any two triangles. Since $K^2=3$, the
position of these two triangles must be one of the following two
cases. (1) They have a common edge and other edges don't intersect.
(2) Each edge of $\Delta$ intersects with exactly one edge of
$\Delta'$. So we just check two special triangles in above cases.
what remains to do is a direct checking. \par
   From above we can easily write down the $45$
(left or right) cosets of $W(F_4)$ in $W(E_6)$.}

\subsubsection{$F_4$ bundles and rational
surfaces}\label{F4-surface}
For $G=F_4$ we take $G'=E_6$, such that $F_4=(E_6)^{Out(E_6)}$.\par
   Let $S=X_6(x_1,\cdots,x_6)$ be a surface
with an $E_6$-configuration (Figure 7), that is, $S$ is a blow-up of
$\bbP^2$ at $6$ points $x_1,\cdots,x_6\in\Sigma$, where
$\Sigma\in|-K_S|$. Take the simple root system $\Delta(E_6)$ and
$\rho\in Out(E_6)$ just as in Section~\ref{E6}. \par
   Once a simple root system is fixed, the restriction from $S$ to
   $\Sigma$ induces a homomorphism $u\in Hom(\Lambda(E_6),\Sigma)$.

\begin{lemma}\label{F4-f} Let $u\in Hom(\Lambda(E_6),\Sigma)$ be an element corresponding to a pair $(S,\Sigma)$,
where $S$ is a surface with an $E_6$-configuration. Then $\rho
\cdot u=u$ if and only if $x_1+x_6=x_2+x_5=x_3+x_4$.\end{lemma}

\begin{proof} Since $u$ is induced by the restriction to $\Sigma$,
$u(\alpha_1)=\calO(l_1-l_2)|_{\Sigma}=x_1-x_2$,
$u(\alpha_2)=x_2-x_3$, $u(\alpha_5)=x_4-x_5$, $u(\alpha_6)=x_5-x_6$.
Therefore $\rho\cdot u=u$ $\Leftrightarrow$ $
u(\alpha_1)=u(\alpha_6),u(\alpha_2)=u(\alpha_5)$ $\Leftrightarrow$ $
 x_1+x_6=x_2+x_5=x_3+x_4$. \end{proof}

Denote $\calS(\Sigma,E_6)$ the moduli space of $G'=E_6$-surfaces
\cite{LZ} with a fixed anti-canonical curve $\Sigma$, and
$\overline{\calS(\Sigma,E_6)}$ the natural compactification by
including all rational surfaces with
$E_6$-configurations. From  \cite{LZ} we know that there is an isomorphism
$\phi:\overline{\calS(\Sigma,E_6)}\isomap\mathcal {M}_{\Sigma}^{E_6}$. Thus we have

\begin{corollary}\label{F4-bundle}  For $u\in\mathcal{M}_{\Sigma}^{F_4}\subset (\mathcal {M}_{\Sigma}^{E_6})^{Out(E_6)}$,
$\phi^{-1}(u)\in\overline{\mathcal {S}(\Sigma,E_6)}$ represents a
class of surfaces $X_6(x_1,\cdots,x_6)$ with
$x_1+x_6=x_2+x_5=x_3+x_4$. \end{corollary}

   Denote $S=X_6'(x_1,\cdots,x_6)$  the blow-up of $\bbP^2$ at
$6$  points $x_1,\cdots,x_6$ on $\Sigma$ which satisfies the
condition  $x_1+x_6=x_2+x_5=x_3+x_4$, with corresponding exceptional
classes $l_1,\cdots,l_6$.
 The condition $x_1+x_6=x_2+x_5=x_3+x_4:=p$ implies that the three lines
$L_{16},L_{25}\mbox{ and }L_{34}$ in $\bbP^2$ intersect at one
points $-p\in \Sigma$, where $L_{ij}$ means the line in $\bbP^2$
passing through these two points $x_i$ and $x_j$. So after blowing
up $\mathbb{P}^2$ at $x_i\in\Sigma,1\leq i\leq 6$, the three $(-1)$
curves $h-l_1-l_6,h-l_2-l_5$ and $h-l_3-l_4$ intersect at one points
$-p\in \Sigma$. So they form a special triangle (see
Section~\ref{E6}). As before, we give the following definition.

\begin{definition}\label{F4-config}{\rm An {\it $F_4$-exceptional system} on
$S=X_6'$ is a $6$-tuple $(e_1,\cdots,e_6)$ consisting of $6$
exceptional divisors which are pairwise disjoint, such that
$y_1+y_6=y_2+y_5=y_3+y_4$, where
$\calO_{\Sigma}(y_i)=\calO(e_i)|_{\Sigma}$. And an {\it
$F_4$-configuration} $\zeta_{F_4}=(e_1,\cdots,e_6)$ just means an
$F_4$-exceptional system on $S$ such that we can consider $S$ as a
blow-up of $\bbP^2$ at $6$ points $y_1,\cdots,y_6$ with
corresponding exceptional divisors $e_1,\cdots,e_6$. For
$S=X_6'(x_1,\cdots,x_6)$, when $x_1,\cdots,x_6$ are in general
position, any $F_4$-exceptional system on $S$ consists of
exceptional curves. Such a surface is called an {\it
$F_4$-surface}.}\end{definition}

 So an $F_4$-surface is automatically an $E_6$-surface (namely, a del Pezzo surface of degree $3$).
 And any $F_4$-exceptional system on an $F_4$-surface is always an
 $F_4$-configuration. See Figure 8 for an $F_4$-configuration. \par
 According to the discussions in Section~\ref{E6},
 the Weyl group $W(F_4)$ is the automorphism group of the sub-root
system of type $D_4$ with simple roots
$\{l_1-l_6,l_2-l_5,l_3-l_4,h-l_1-l_2-l_3\}$, and $W(F_4)\cong
S_3\rtimes W(D_4)$. Therefore we have

\begin{lemma}\label{W(F4)action} (i) Let $S=X_6'$ be a surface with an
$F_4$-configuration. Then the Weyl group $W(F_4)$ acts on all
$F_4$-exceptional systems on $S$ simply transitively.\par (ii)
Moreover, if $S$ is an $F_4$-surface, then the Weyl group $W(F_4)$
acts on all $F_4$-configurations on $S$ simply transitively.\end{lemma}

 \begin{proposition}\label{F4-moduli} Let $\mathcal {S}(\Sigma,F_4)$ be the moduli space
of pairs $(X_6',\Sigma)$ where $X_6'$ is an $F_4$-surface containing
$\Sigma$ as an anti-canonical curve, and
$\mathcal{M}_{\Sigma}^{F_4}$ be the moduli space of flat $F_4$
bundles over $\Sigma$. Then we have \par

(i) $\mathcal {S}(\Sigma,F_4)$ is embedded into
$\mathcal{M}_{\Sigma}^{F_4}$ as an open dense subset. \par

(ii) Moreover, this embedding can be extended naturally to an
isomorphism $$\overline{\mathcal {S}(\Sigma,F_4)}\cong
\mathcal{M}_{\Sigma}^{F_4},$$ by including all  rational
surfaces with $F_4$-configurations.
\end{proposition}

\begin{proof} Firstly, we can take the simple root system of $F_4$ as
$$\Delta(F_4)=\{\beta_1,\beta_2,\beta_3,\beta_4\},$$ where $\beta_1=l_1-l_2+l_5-l_6$, $\beta_2=l_2-l_3+l_4-l_5$,
$\beta_3=2(h-l_1-l_2-l_3)$, and $\beta_4=2(l_3-l_4)$.\par
  Secondly, the restriction to $\Sigma$ induces the following system of linear equations:
  $$\left \{ \begin{array}{l}  x_1-x_2+x_5-x_6=p_1,\\
       x_2-x_3+x_4-x_5=p_2,\\
       2(-x_1-x_2-x_3)=p_3,\\
       2(x_3-x_4)=p_4,\\
       x_1+x_6=x_2+x_5=x_3+x_4.
       \end{array}
\right. $$
  Since the determinant is non-zero, the result follows by the same
argument as in $B_n$ case. \end{proof}

   The Lie algebra bundle of type $F_4$ over $X_6'$ can be
   constructed as (for brevity, we denote
   $\varepsilon_1=l_2-l_3+l_4-l_5$, $\varepsilon_2=l_2+l_3-l_4-l_5$, $\varepsilon_3=2h-2l_1-l_2-l_3-l_4-l_5$, and
   $\varepsilon_4=2h-2l_6-l_2-l_3-l_4-l_5$)

     $$\mathscr{F}_4=\mathcal {O}^{\bigoplus 4}\bigoplus_{D\in R(F_4)}\mathcal{O}(D),$$
where $R(F_4)$ is the root system of $F_4$:
$$R(F_4)=\{\pm \varepsilon_i,\ \pm (\varepsilon_i\pm \varepsilon_j),\ \pm {1\over 2}(\varepsilon_1
\pm \varepsilon_2\pm \varepsilon_3\pm \varepsilon_4)\ |\ i\neq j\}.$$

{\bf Remark} {\rm The $27$ lines determine the $27$-dimensional
fundamental representation of $E_6$. Restricted to $\Sigma$, they
give us a representation bundle of rank $27$ (of $\mathscr{F}_4$)
over $\Sigma$. The weights associated to the $3$ special lines
$h-l_1-l_6,h-l_2-l_5,h-l_3-l_4$ restrict to zero and these $3$
weights add to zero before restriction (since
$(h-l_1-l_6)+(h-l_2-l_5)+(h-l_3-l_4)=-K$). The remaining $24$
weights associated to other $24$ lines restrict to the $24$ short
roots of $\mathscr{F}_4$. The $24$ lines  and a rank $2$ bundle $V$
determine the $26$-dimensional irreducible fundamental
representation $U$ of $\mathscr{F}_4$. Here $V$ is determined as
follows. Since
$\mathcal{O}_{\Sigma}(h-l_1-l_6)=\mathcal{O}_{\Sigma}(h-l_2-l_5)=\mathcal{O}_{\Sigma}(h-l_3-l_4)=\mathcal{O}_{\Sigma}((-p))$,
taking the trace, we have the following exact sequence:
$$0\rightarrow ker(tr)\rightarrow \mathcal{O}_{\Sigma}((-p))^{\bigoplus 3}\rightarrow \mathcal{O}_{\Sigma}((-p))\rightarrow 0.$$
Then we take $V=ker(tr)$.}

   For more details on the $26$-dimensional fundamental representation of $F_4$, one can consult \cite{Adams}.

\section{Conclusion}\label{end}

   We summarize our results in \cite{LZ} and this paper as follows.  Let $\Sigma$ be a fixed elliptic curve with
   identity $0\in \Sigma$. Let $G$ be any compact, simple and simply connected Lie groups, simply laced or not.
    Denote $\calS(\Sigma,G)$ the moduli space of $G$-surfaces containing a fixed anti-canonical curve $\Sigma$.
   Denote $\calM_{\Sigma}^{G}$ the moduli space of flat $G$ bundles over $\Sigma$. Then we have

\begin{theorem}(i) We can construct Lie algebra $Lie(G)$-bundles over
each $G$-surface.\par
        (ii) The restriction of these Lie algebra bundles to the anti-canonical curve $\Sigma$ induces an embedding of
 $\calS(\Sigma,G)$  into $\calM_{\Sigma}^{G}$   as an open dense subset.\par
       (iii) This embedding can be extended to an isomorphism from $\overline{\calS(\Sigma,G)}$
 onto $\calM_{\Sigma}^{G}$, where $\overline{\calS(\Sigma,G)}$ is a natural and explicit compactification
 of $\calS(\Sigma,G)$, by including all  rational surfaces
 with $G$-configurations.\end{theorem}

   {\bf Remark}  {\rm (i) The result is known for $G=E_n (n=6,7,8)$ case (see \cite{Donagi}\cite{Donagi2}\cite{FMW1}).\par
   (ii) We have mentioned in the beginning of  $\S$ 1 that there is another reduction of the
   non-simply laced cases to simply laced cases. In fact, using this reduction, we will obtain the same result,
   just following the steps as above.}

   According to a theorem of
   \cite{Looijenga}, the moduli space
   $\calS(\Sigma,G)$ is a weighted projective space. Thus the compactification
   $\overline{\calS(\Sigma,G)}$ is a weighted projective space.
   Conversely, we believe that the above identification between $\calS(\Sigma,G)$ and
   $\overline{\calS(\Sigma,G)}$ will give us another proof for
   this theorem. This is already done in $E_n$ case by
   \cite{Donagi}\cite{Donagi2}\cite{FMW1} and so on.\\

\noindent{\bf Acknowledgements}. The second author would like to
express his gratitude to the Institute of Mathematics of Johannes
Gutenberg University Mainz, Germany and Kang Zuo for their support.
 He would also like to express his gratitude to Changzheng Li and Xiaowei Wang.

\end{document}